\documentclass[12pt,a4paper,reqno]{amsart}
\usepackage[latin1]{inputenc}
\usepackage{amsmath}
\usepackage{amsfonts}
\usepackage{amssymb}
\usepackage{amsmath,bbm,amssymb,amsxtra}
\usepackage{enumerate}
\allowdisplaybreaks
\usepackage{epsfig}
\usepackage{ae}
    \newcommand{\refeq}[1]{(\ref{#1})}

    \def\Xint#1{\mathchoice
    {\XXint\displaystyle\textstyle{#1}}%
    {\XXint\textstyle\scriptstyle{#1}}%
    {\XXint\scriptstyle\scriptscriptstyle{#1}}%
    {\XXint\scriptscriptstyle\scriptscriptstyle{#1}}%
    \!\int}
    \def\XXint#1#2#3{{\setbox0=\hbox{$#1{#2#3}{\int}$ }
    \vcenter{\hbox{$#2#3$ }}\kern-.6\wd0}}
    \def\dashint{\Xint-}

    \newcommand{\closure}[2][3]{{}\mkern#1mu\overline{\mkern-#1mu#2}}

\theoremstyle{definition}

\newtheorem{lemma}{Lemma}[section]

\newtheorem{proposition}[lemma]{Proposition}

\newtheorem{theorem}[lemma]{Theorem}
\newtheorem{corollary}[lemma]{Corollary}

\newtheorem*{remark2}{Remark}
\newtheorem{definition}[lemma]{Definition}

\newcommand{\prop}[1]{\begin{proposition}\label{#1}
\sl }
\newcommand{\eprop}{\end{proposition}}
\newcommand{\thm}[1]{\begin{theorem}\label{#1}
\sl }
\newcommand{\ethm}{\end{theorem}}

\newcommand{\lem}[1]{\begin{lemma}\label{#1}
\sl }
\newcommand{\elem}{\end{lemma}}

\newcommand{\defin}[1]{\begin{definition}\label{#1}
\sl }
\newcommand{\edefin}{\end{definition}}

\newcommand{\beqno}{\begin{eqnarray*}}
\newcommand{\eeqno}{\end{eqnarray*}}
\newcommand{\beqla}[1] {\begin {eqnarray}\label{#1}}
\def\eeq {\end {eqnarray}}
\newcommand{\beq}{\begin {eqnarray}}

\newcommand{\real}{{\mathbb R}}

\newcommand{\integer}{{\mathbb Z}}

\newcommand{\diam}{{\rm diam}\,}
\newcommand{\supp}{{\rm supp}\,}
\newcommand{\Lip}{{\rm Lip}\,}

\title[Besov compositions]{Bounded compositions on scaling invariant Besov spaces}

\author[Koch]{Herbert Koch}
\address{Mathematisches Institut, Universit\"at Bonn,
Endenicher Allee 60
D - 53115 Bonn, Germany}
\email{koch@math.uni-bonn.de}

\author[Koskela]{Pekka Koskela}
\address{Department of Mathematics and Statistics, University of
Jyv\"askyl\"a, PO~Box~35, FI-40014 Jyv\"askyl\"a, Finland}
\email{pkoskela@maths.jyu.fi}

\author[Saksman]{Eero Saksman}
\address{Department of Mathematics and Statistics, University of
Helsinki, PO~Box~68, FI-00014 Helsinki, Finland}
\email{eero.saksman@helsinki.fi}

\author[Soto]{Tom\'as Soto}
\address{Department of Mathematics and Statistics, University of
Helsinki, PO~Box~68, FI-00014 Helsinki, Finland}
\email{tomas.soto@helsinki.fi}

\thanks{The second author was supported by the Academy of Finland, project 131477}
\thanks{The third and fourth authors were supported by the Finnish CoE in Analysis and Dynamics Research}



\begin{document}

\maketitle

\begin{abstract}
For $0 < s < 1 < q < \infty$, we characterize the homeomorphisms $\varphi : \real^n \to \real^n$ for which the composition operator $f \mapsto f \circ \varphi$ is bounded on the homogeneous, scaling invariant Besov space $\dot{B}^s_{n/s,q}(\real^n)$, where the emphasis is on the case $q\not=n/s$, left open in the previous literature. We also establish an analogous result for Besov-type function spaces on a wide class of metric measure spaces as well, and make some new remarks considering the scaling invariant Triebel-Lizorkin spaces $\dot{F}^s_{n/s,q}(\real^n)$ with $0 < s < 1$ and $0 < q \leq \infty$.
\end{abstract}

\section{Introduction}

Let us first recall the definition in terms of first-order differences of the (homogeneous) Besov space $\dot{B}_{p,q}^s(\real^n)$, $0<s<1$, $1 \leq p,\; q \leq \infty$: it is the collection of measurable functions $f$ (modulo functions which are constant almost everywhere) for which the norm
\beqla{besov-norm}
  \|f\|_{\dot{B}_{p,q}^s(\real^n)} := \left(\int_{\real^n}|h|^{-qs} \left( \int_{\real^n} |f(x+h)-f(x)|^p dx \right)^{q/p} \frac{dh}{|h|^n}  \right)^{1/q}
\eeq
(with the usual modifications for $p = \infty$ and $q = \infty$) is finite. The norm defines a Banach space which coincides with the usual Fourier analytically defined space; see \cite[Section 5.2.3]{T}. For $0 < \min(p,q) < 1$, $\dot{B}_{p,q}^s(\real^n)$ will stand for the analogous Fourier analytically defined quasi-Banach space of distributions.

The main question studied in  the present note asks for which homeomorphisms $\varphi$ of $\real^n$ is the mapping\footnote{In our proofs, we shall only deal with continuous functions, so one may think of the mapping to be defined by $f \mapsto f\circ \varphi$ for continuous $f$, and for the rest of the space by completion. Alternatively, the mapping $\varphi$ may be assumed to have Lusin's condition (N) so that the mapping $f \mapsto f\circ \varphi$ is well defined in the sense of equivalence classes with respect to equality almost everywhere.}
\[
  f \mapsto f\circ \varphi
\]
bounded on $\dot{B}_{p,q}^s(\real^n)$? Especially, we are interested in the case where
the norm \refeq{besov-norm} is invariant under scaling changes of variables, i.e. in the case $p = n/s$. Namely, this kind of ``conformal'' invariance makes one expect the class of admissible $\varphi$ to be rich and, in particular, that it might contain all quasiconformal maps, as this is almost obviously true for the conformally invariant Sobolev space $\dot{W}^{1,n}(\real^n).$

For $\dot{B}_{n/s,n/s}^s(\real^n)$ with $0 < s \leq n/(n+1)$, this question is answered in \cite{Vo} and \cite{BS}: it turns out that the quasiconformality (see \cite{Va} and \cite{K} for the definition and general properties of such mappings) of $\varphi$ is a necessary and sufficient condition. Also the Besov spaces with $p \neq n/s$ are dealt with in both papers. One should note that the invariance under quasiconformal maps for said diagonal spaces can be obtained by interpolation.\footnote{For example, one can first interpolate (see e.g.~\cite[Corollary 8.3]{FJ}) between Reimann's result on the space $BMO(\real^n)$ \cite{R} and the ``easy'' case $\dot{W}^{1,n}(\real^n)$ to obtain boundedness on the Triebel-Lizorkin space $\dot{F}^s_{n/s,2}(\real^n)$ for all $s \in (0,1)$, and then interpolate (see e.g.~\cite[Theorem 2.4.2/1]{T2}) between two such spaces to obtain boundedness on $\dot{F}_{n/s,n/s}^s(\real^n)=\dot{B}_{n/s,n/s}^s(\real^n)$ for all $0 < s < 1$.}
The quasiconformal invariance phenomenon carries over to a large class of Triebel-Lizorkin spaces: it was recently shown in \cite{KYZ2} that for $n \geq 2$, $0 < s < 1$ and $q > n/(n+s)$, the space $\dot{F}_{n/s,q}^s(\real^n)$ is invariant under quasiconformal mappings, yielding another proof for the quasiconformal invariance of $\dot{B}_{n/s,n/s}^s(\real^n)=\dot{F}_{n/s,n/s}^s(\real^n)$ for all $0 < s < 1$.

However, the case of Besov spaces with $p = n/s \neq q$ appears to be completely open and it is the aim of this paper to fill in this gap. In fact, in section \ref{section:real} we shall prove the following somewhat surprising result.

\thm{th:bilip-necessity}
Suppose that $\varphi : \real^n \to \real^n$, $n \geq 2$, is a homeomorphism such that the mapping induced by $\varphi$ is bounded on $\dot{B}_{n/s,q}^s(\real^n)$ for some $0 < s < 1$ and $1 < q < \infty$, $q \neq n/s$. Then $\varphi$ is bi-Lipschitz.
\ethm

The converse is obvious: all Besov spaces in the preceding statement are invariant under bi-Lipschitz mappings. By combining Theorem \ref{th:bilip-necessity} with the known results in the case $p=q=n/s$ we thus have the following conclusion.

\begin{corollary}\label{boundedness-conditions}\sl
Suppose that $\varphi : \real^n \to \real^n$, $n \geq 2$, is a homeomorphism. Then

(i) if $0 < s < 1$, the mapping induced by $\varphi$ is bounded on $\dot{B}_{n/s,n/s}^s (\real^n)$ if and only if $\varphi$ is quasiconformal;

(ii) if $0 < s < 1 < q < \infty$ and $q \neq n/s$, the mapping induced by $\varphi$ is bounded on $\dot{B}_{n/s,q}^s (\real^n)$ if and only if $\varphi$ is bi-Lipschitz.
\end{corollary}

We also make some remarks concerning the case $q \leq 1$. For an analytic approach in the spirit of Reimann \cite{R}, see the discussion preceding Lemma \ref{le:small-q}. For a localized approach in the spirit of Astala \cite{A}, see the discussion preceding Theorem \ref{th:small-q-3}. 

In section \ref{section:tl} we establish a converse to \cite[Theorem 1.1]{KYZ2}, i.e.~the necessity of the quasiconformality of $\varphi$ for the mapping induced by $\varphi$ to be bounded on a scaling invariant Triebel-Lizorkin space.

A similar question in the setting of metric spaces is studied in \cite{B} and \cite{BP}, where it is shown (among other things) for compact metric spaces $Z_1$ and $Z_2$ which are respectively Ahlfors regular, and satisfy some other reasonable assumptions, that a homeomorphism $\varphi : Z_1 \to Z_2$ induces an isomorphism between the respective diagonal Besov spaces if and only if $\varphi$ is quasisymmetric.

In the proofs of our main results we shall work with characterizations of Besov-type function spaces on doubling metric measure spaces introduced in \cite{GKZ}, and for this reason our methods also yield extensions in a sizeable class of metric spaces considered in \cite{HK} (see the discussion preceding Proposition \ref{pr:qc-necessity-m} for the precise assumptions). Section \ref{section:metric} is devoted to these extensions.

We finally list some notation conventions. On any metric space $(X,d)$, we write $B_d(x,r)$ (or just $B(x,r)$ if the metric is obvious from the context) for the ball with center $x \in X$ and radius $r > 0$: $\{y \in X : d(x,y) < r\}$, and $\closure[2]{B}_d(x,r)$ or $\closure[2]{B}(x,r)$ for the corresponding closed ball. Similarly, for $x \in X$ and $0 < r < R$, $A_d(x,r,R)$ or just $A(x,r,R)$ stands for the annulus $\{y \in X : r < d(y,x) < R\}$. For a set $A$ of finite measure with respect to a measure $\mu$, we use the notation $\dashint_A f d\mu$ for the average of $f$ over $A$, i.e.~$\mu(A)^{-1}\int_A f d\mu$ whenever the latter quantity is well defined. In the context of $\real^n$, $|A|$ stands for the $n$-Lebesgue measure of a measurable set $A$. For two non-negative functions $f$ and $g$ with the same domain, the notation $f \lesssim g$ stands for $f \leq Cg$ for some positive constant $C$, usually independent of some parameters. The notation $f \approx g$ means that $f \lesssim g$ and $g \lesssim f$.

\section{Besov spaces on $\real^n$}\label{section:real}

In this section, we will make use of the following characterizations of the Besov (quasi-)norm, see \cite{GKZ} and \cite{KYZ2}:
\[
  \| f\|_{\dot{B}^{s}_{p,q}(\real^n)} \approx \left( \int_{0}^{\infty} [t^{-s} C_p(f)(t)]^q \frac{dt}{t} \right)^{1/q}
\]
for $0 < s < 1$, $n/(n+s) < p \leq \infty$ and $0 < q \leq \infty$, where
\[
  C_p (f)(t) = \left( \int_{\real^n} \dashint_{B(x,t)} |f(x) - f(y)|^p dy dx \right)^{1/p}
\]
(usual modifications for $p = \infty$ and $q = \infty$). In our applications we always have $p = n/s > n/(n+s)$ so that the characterization above applies. Another characterization of the norm for the same parameter range is the following (we again refer to \cite{GKZ} and \cite{KYZ2}):
\beq \label{eq:besov-characterizations}
  \| f\|_{\dot{B}^{s}_{p,q}(\real^n)} \approx \inf_{\overrightarrow{g}}\left( \sum_{k \in \integer} \|g_k \|_{L^p(\real^n)}^q \right)^{1/q},
\eeq
where the infimum is taken over the sequences $\overrightarrow{g} := (g_k)_{k\in\integer}$ (the so-called \emph{Haj\l asz $s$-gradients} of $f$) of measurable functions $g_k : \real^n \to [0,\infty]$ satisfying
\beq \label{eq:besov-characterizations-2}
  |f(x) - f(y)| \leq |x - y|^s \left(g_k(x) + g_k(y) \right)
\eeq
for all $k \in \integer$ and all $x$, $y$ in $\real^n$ with $2^{-k-1} \leq |x - y| < 2^{-k}$. Note that the quantity on the right-hand side of \refeq{eq:besov-characterizations} is monotonic with respect to $q$.

The dimension $n$ will always be at least $2$.

We start out with several auxiliary results which will be needed in the proofs of Proposition \ref{pr:qc-necessity} and Theorem \ref{th:bilip-necessity}. The first two lemmas estimate the Besov norms of certain types of functions which will appear in our constructions. Essentially, for suitably chosen bump functions $h_j$ supported on disjoint balls $B_j$, and all sequences of positive reals $(b_j)_{j \geq 0}$, the $\dot{B}_{n/s,q}^s(\real^n)$ norm of $\sum_j b_j h_j$ turns out to be comparable to $\|(b_j)_j\|_{\ell^q}$ if the radii of $B_j$ form a dyadic scale (Lemma \ref{le:balls-different-sizes}), and to $\|(b_j)_j\|_{\ell^{n/s}}$ if the radii are equal (Lemma \ref{le:balls-same-sizes}).

\lem{le:balls-different-sizes}

Let $0 < s < 1$, $1 \leq q \leq \infty$ and $0 < R < \infty$. Suppose that $(B_j)_{j = 0}^\infty$ is a sequence of balls in $\real^n$ such that $B_j$ has radius $2^{-j} R$ for all $j \geq 0$, and that $(f_j)_{j=0}^\infty$ is a sequence of measurable functions on $\real^n$ with $\supp f_j \subset \closure[2]{B_j}$ for all $j \geq 0$. Write $F := \sum_{j\geq 0} f_j$.

(i) If the functions $f_j$ are Lipschitz with
\[
  \Lip f_j \leq 2^j R^{-1} b_j, \quad b_j > 0,
\]
for all $j \geq 0$, we have
\[
  \|F\|_{\dot{B}_{n/s,q}^{s}(\real^n)} \lesssim \| (b_j)_{j \geq 0} \|_{\ell^q}.
\]

(ii) If the balls $9 B_j$ are pairwise disjoint, we have
\[
  \|F\|_{\dot{B}_{n/s,q}^{s}(\real^n)} \gtrsim \left\| (2^{js} R^{-s} \|f_j\|_{L^{n/s}(\real^n)})_{j\geq 0} \right\|_{\ell^q}.
\]
The implied constants in both conclusions are independent of $R$.
\elem

\begin{proof}
Let $p := n/s$ and suppose for the moment that $1 \leq q < \infty$. We first consider the situation in (i). The Lipschitz continuity of each $f_j$ gives $\|f_j\|_{L^{\infty}} \lesssim b_j$, so their $L^p$ norms can be estimated by
\[
  \|f_j\|_{L^p} \lesssim b_j |B_j|^{1/p} \approx (2^{-j} R)^s b_j.
\]

We shall now estimate $C_p(F)(t)$ for different ranges of $t$. For $t \geq R$, there is the simple estimate
\[ 
  C_p(F)(t)
  \leq \left( \int_{\real^n} \left[ \left( \dashint_{B(x,t)} |F(x)|^p dy \right)^{1/p}
  + \left( \dashint_{B(x,t)} |F(y)|^p dy \right)^{1/p} \right]^{p} dx \right)^{1/p}
  \lesssim \|F\|_{L^p},
\]
and by H\"older's inequality,
\[
  \|F\|_{L^p} \leq \sum_{j \geq 0} \|f_j\|_{L^p} \lesssim R^{s} \sum_{j \geq 0} 2^{-js} b_j
  \lesssim R^s \|(b_j)_{j\geq 0}\|_{\ell^q}.
\]

Suppose now that $2^{-k} R \leq t <  2^{1-k} R$ for a positive integer $k$. For any $j \geq k$, we have $C_p (f_j)(t) \lesssim (2^{-j} R)^s b_j$ as above. For any $j < k$, we may use the Lipschitz continuity of $f_j$ to obtain 
\begin{align*}
  C_p(f_j)(t)
& \leq 2^{j} R^{-1} b_j \left( \int_{2 B_j} \dashint_{B(x,t)} |x - y|^p dy dx \right)^{1/p}
  \approx 2^{j} R^{-1} b_j t |2B_j|^{1/p}
  \approx 2^{(1-s)j} 2^{-k} R^s b_j.
\end{align*} 
Putting these estimates together, we have
\[
  C_p(F)(t) \leq \sum_{j \geq 0} C_p (f_j)(t)
  \lesssim R^s 2^{-k} \sum_{j=0}^{k-1} 2^{(1-s)j}b_j + R^s \sum_{j \geq k} 2^{-js} b_j .
\]

Finally,
\[
  \|F\|_{\dot{B}^{s}_{p,q}(\real^n)}^q \approx \int_{0}^{R} t^{-1-qs} [C_p(F)(t)]^q dt
  + \int_{R}^{\infty} t^{-1-qs} [C_p(F)(t)]^q dt =: I + J.
\]
Here $J \lesssim R^{qs} \|(b_j)_{j\geq 0}\|_{\ell^q}^q \int_{R}^{\infty} t^{-1-qs} dt
  \approx \|(b_j)_{j\geq 0}\|_{\ell^q}^q$, and for $I$ we use the previous estimates on each interval $[2^{-k}R , 2^{1-k} R)$, $k \geq 1$, to get
\[
  I \lesssim \sum_{k = 1}^{\infty} 2^{sqk}
  \left( 2^{-k} \sum_{j=0}^{k-1} 2^{(1-s)j}b_j + \sum_{j \geq k} 2^{-js} b_j \right)^q.
\]
Denoting by $r_k^q$ the $k$th term in the outer sum and writing $\tilde{s} := \min(s,1-s)$, we get
\[
  r_k = \sum_{j=0}^{k-1} 2^{(1-s)(j-k)}b_j + \sum_{j \geq k }2^{(k-j)s} b_j 
  \leq \sum_{j \geq 0} 2^{-|k-j|\tilde{s}} b_j =: r'_k.
\]
Now the convolution operator $(b_j)_{j \geq 0} \mapsto (r'_k)_{k \geq 1}$ is bounded on the space $\ell^q$, so we actually get
\[
  I \lesssim \sum_{k \geq 1} (r'_k)^q \lesssim \sum_{j \geq 0} b_j^q
= \|(b_j)_{j \geq 0}\|_{\ell^q}^q,
\]
which finishes the proof of part (i).

Suppose now that we are in the situation of part (ii). The disjointness of the balls $9 B_j$ shows that $ |F(x) - F(y)| = |f_j(x) - f_j(y)|$ for any $x$ and $y$ in $9 B_j$. Also, for any $x \in \real^n$ and $\lambda > 0$, the measure of the annulus $A(x,\lambda/2, \lambda)$ is comparable to the measure of the ball $B(x,\lambda)$. In particular, if $t \in (2^{2-j} R, 2^{3-j} R]$ for some integer $j \geq 0$, one has
\[
  C_p(F)(t) \gtrsim \left( \int_{B_j} \dashint_{A(x,t/2,t)} |f_j(x)-f_j(y)|^p dydx \right)^{1/p}
= \|f_j\|_{L^p(\real^n)}.
\]
Using this estimate on each interval $(2^{2-j} R , 2^{3-j} R]$ yields
\[
  \int_{0}^{\infty} t^{-1-qs} C_p (F)(t)^q dt \gtrsim
  \sum_{j \geq 0} \|f_j\|_{L^p(\real^n)}^q \int_{2^{2-j} R}^{2^{3-j} R} t^{-qs} \frac{dt}{t}
  \approx \sum_{j \geq 0} \left( 2^{js} R^{-s} \|f_j\|_{L^p(\real^n)} \right)^q.
\]

Finally, both results for $q = \infty$ are obtained by an easy modification of the proof above and by noting that
\[
  \|F\|_{\dot{B}^s_{p,\infty}(\real^n)} \approx \sup_{k \in \integer} \left(2^{-k}R\right)^{-s} C_p(F)(2^{-k}R)
\]
with the implied constants independent of $R$.
\end{proof}

\lem{le:balls-same-sizes}

Let $0 < s < 1$, $1 \leq q \leq \infty$ and $0 < R < \infty$. Suppose that $(B_j)_{j=0}^\infty$ is a collection of balls in $\real^n$ such that each ball has radius $R$ and the balls $9 B_j$ are pairwise disjoint. Suppose also that $(f_j)_{j=0}^\infty$ is a sequence of functions on $\real^n$ with $\supp f_j \subset \closure[2]{B_j}$ for all $j \geq 0$. Write $F := \sum_{j \geq 0} f_j$.

(i) If the functions $f_j$ are Lipschitz with
\[
  \Lip f_j \leq R^{-1} b_j, \quad b_j > 0,
\]
for all $j \geq 0$, we have
\[
  \|F\|_{\dot{B}_{n/s,q}^s(\real^n)} \lesssim \| (b_j)_{j \geq 0} \|_{\ell^{n/s}}.
\]

(ii) We have (with the functions $f_j$ not necessarily Lipschitz)
\[
  \|F\|_{\dot{B}_{n/s,q}^s(\real^n)} \gtrsim \left\| (R^{-s}\|f_j\|_{L^{n/s}(\real^n)})_{j\geq 0}  \right\|_{\ell^{n/s}}.
\]
The implied constants in both conclusions are independent of $R$.
\elem

\begin{proof}
The proof goes along the same lines as in the previous lemma. Again, we denote $n/s$ by $p$ and only consider the case $1 \leq q < \infty$. In the situation of (i), we have $\|f_j\|_{L^p(\real^n)} \lesssim R^s b_j$ for all $j \geq 0$, and for $t \geq R$ we therefore obtain
\[
  C_p(F)(t) \lesssim \|F\|_{L^p(\real^n)}
  = \left( \sum_{j \geq 0} \|f_j\|_{L^p}^p \right)^{1/p}
  \lesssim R^s \| (b_j)_{j \geq 0} \|_{\ell^p}.
\]
For $t < R$, we use the Lipschitz continuity of each $f_j$ to obtain
\[
  C_p(f_j)(t) \leq R^{-1} b_j \left( \int_{2B_j} \dashint_{B(x,t)} |x-y|^p dy dx\right)^{1/p}
  \lesssim R^{-1} b_j t |2 B_j|^{1/p} \lesssim R^{s-1} t b_j.
\]
The disjointness of the balls $9 B_j$ thus yields
\[
  C_p(F)(t)^p = \sum_{j \geq 0} \int_{\real^n} \dashint_{B(x,t)} |f_j(x) - f_j(y)|^p dy dx
  = \sum_{j \geq 0} C_p (f_j)(t)^p \lesssim (R^{s-1} t)^p \sum_{j \geq 0} b_j^p,
\]
so that
\[
  \|F\|_{\dot{B}_{p,q}^{s}(\real^n)}^q \lesssim \| (b_j)_{j \geq 0} \|_{\ell^p}^q
  \left( (R^{s-1})^q \int_{0}^{R}t^{-1-qs+q} dt
    + R^{qs} \int_{R}^{\infty} t^{-1-qs}dt  \right)
  \approx \| (b_j)_{j \geq 0} \|_{\ell^p}^q.
\] 

In the situation of (ii), we will simply estimate $C_p (F)(t)$ from below for $t \in (4R, 8R]$. For these $t$, the disjointness of the balls $9 B_j$ yields
\[
  C_p(F)(t)^p \geq \sum_{j \geq 0} \int_{B_j} \dashint_{B(x,t)} |f_j(x) - f_j(y)|^p dy dx
  \gtrsim \sum_{j \geq 0} \| f_j \|_{L^p(\real^n)}^{p}
\]
as in the proof of part (ii) of the previous lemma. Therefore,
\[
  \|F\|_{\dot{B}_{p,q}^{s}(\real^n)}^q \gtrsim \left\| (\|f_j\|_{L^p(\real^n)} )_{j \geq 0} \right\|_{\ell^p}^q
  \int_{4R}^{8R} t^{-1-qs} dt \approx 
  \left\| (R^{-s}\|f_j\|_{L^{p}(\real^n)})_{j\geq 0}  \right\|_{\ell^{p}}^q. \qedhere
\]
\end{proof}

Here is an estimate which we will use to get a lower bound for the Besov capacity of a quasi-annulus later on. A similar result with the $L^p$ norm of an upper gradient instead of the Besov norm is contained in \cite[Theorem 5.9]{HK}; the proof below is a modification of the proof therein.

\lem{le:capacity-lower}
Suppose that $0 < s < 1$, $1 \leq q < \infty$ and that $B_R$ is a ball of radius $R$ in $\real^n$. If $0 < \lambda \leq 1$, $E$ and $F$ are two such compact and connected subsets of $B_R$ that
\[
  \min(\diam E, \diam F) \geq \lambda R
\]
and $u$ is a continuous function on $\real^n$ with $u_{|E} \leq 0$ and $u_{|F} \geq 1$, we have
\[
  \| u \|_{\dot{B}_{n/s,q}^s(\real^n)} \geq c \lambda^{1/q},
\]
where $c > 0$ depends only on $q$, $s$ and $n$.
\elem

\begin{proof}
We will work with the characterization of the Besov norms through Haj\l asz-type gradients; see \refeq{eq:besov-characterizations}. By the monotonicity of the $\ell^q$ norms with respect to to $q$ (see the discussion in the beginning of this chapter), it suffices to consider the case $q \geq n/s =: p$. Recalling that $|A(x,\lambda/2,\lambda)| \approx \lambda^n$ for all $x \in \real^n$ and $\lambda > 0$, we have $|A_x^j| \approx 2^{-2j}R$ where
$
  A_x^j := A\left( x , 2^{-2j-1}R , 2^{-2j}R \right)
$
for $x \in \real^n$ and $j \geq 0$. Let $(g_k)_{k\in\integer}$ be a Haj\l asz $s$-gradient of $u$ so that
\[
  \left(\sum_k\|g_k\|_{L^p(\real^n)}^q\right)^{1/q} \approx \|u\|_{\dot{B}_{p,q}^s(\real^n)}.
\]
In the rest of this proof, the notation $|x-y| \sim 2^{-k}$, $k\in \integer$, will be used with the meaning $2^{-k-1} \leq |x-y| < 2^{-k}$.

Assume first that there exist points $x \in E$ and $y \in F$ with $| u(x) - u_{A_x^0} | \leq 1/5$ and $| u(y) - u_{A_y^0} | \leq 1/5$. Thus,
\[
  1 \leq |u(x)-u(y)| \leq 2/5 + |u_{A_x^0} - u_{A_y^0} |,
\]
so that $|u_{A_x^0} - u_{A_y^0}| \geq 3/5$. On the other hand,
\begin{align}
|u_{A_x^0} - u_{A_y^0} | & \lesssim R^{-2n} \iint_{A_x^0 \times A_y^0} |u(x') - u(y')| dx'dy' \notag \\
  & \lesssim R^{-2n} \sum_{2^{-k} \leq 4R} \iint_{\substack{ A_x^0 \times A_y^0 \\ |x' - y'| \sim 2^{-k}}}|u(x')-u(y')|dx'dy' \notag \\
  & \lesssim R^{-2n} \sum_{2^{-k} \leq 4R} 2^{-ks} \iint_{\substack{ A_x^0 \times A_y^0 \\ |x' - y'| \sim 2^{-k}}}\left(g_k(x')+g_k(y')\right)dx'dy' \notag \\
  & \lesssim R^{-2n} \sum_{2^{-k} \leq 4R} 2^{-k(n+s)} \int_{2B_R} g_k(z) dz \notag \\
  & \lesssim R^{n/p' - 2n} \sum_{2^{-k} \leq 4R} 2^{-k(n+s)} \| g_k \|_{L^{p}(\real^n)} \label{eq:linf-embedding} \\
  & \lesssim R^{-n-s} R^{n+s} \left( \sum_k\|g_k\|_{L^p(\real^n)}^q\right)^{1/q} \approx \|u\|_{\dot{B}_{p,q}^s(\real^n)}, \notag
\end{align}
whence $\|u\|_{\dot{B}_{p,q}^s(\real^n)} \gtrsim 1$.

To deal with the other case, we may assume by symmetry that $|u(x) - u_{A_x^0}| > 1/5$ for all $x \in E$. By the continuity of $u$ we thus have
\[
  1 \lesssim \sum_{j \geq 0} | u_{A_x^{j}} - u_{A_x^{j+1}} |.
\]
The $j$th term of this series can be estimated as above and by taking into account the geometry of the situation:
\begin{align}
  | u_{A_x^{j}} - u_{A_x^{j+1}} | & \lesssim \left(2^{-2j}R\right)^{-2n} \iint_{A_x^{j}\times A_x^{j+1}} |u(x')-u(y')|dx' dy' \notag \\
& \lesssim \left(2^{-2j}R\right)^{-2n} \sum_{2^{-2j-2}R \leq 2^{-k} \leq 2^{-2j+1}R } 2^{-k(n+s)} \left( \int_{A_x^j} + \int_{A_x^{j+1}} \right) g_k(z) dz \notag \\
& \lesssim \sum_{2^{-2j-2}R \leq 2^{-k} \leq 2^{-2j+1}R } 2^{k(2n - (n+s))} \int_{B(x,2^{-k+2})} g_k(z) dz \notag \\
& \lesssim \sum_{2^{-2j-2}R \leq 2^{-k} \leq 2^{-2j+1}R } \left( \int_{B(x,2^{-k+2})} g_k(z)^p dz \right)^{1/p}. \label{eq:linf-embedding-2}
\end{align}
There are at most four terms in the above sum, so summing over $j$ yields
\[
  1 \lesssim \sum_{2^{-k} \leq 2R} \left( \int_{B(x,2^{-k+2})} g_k(z)^p dz \right)^{1/p},
\]
which in particular means that there exists an $\epsilon > 0$ depending only on the data such that for every $x \in E$ we can find an integer $k_x$ with
\[
  \left( \int_{B(x,2^{-k_x+2})} g_{k_x}(z)^p dz \right)^{1/p} \geq \epsilon \cdot 2^{-(k_x-2)/q} R^{-1/q}.
\]
We may now take a disjoint and countable collection of balls $B_\ell := B(x_\ell,2^{-r_\ell + 2})$ with $r_\ell = k_{x_\ell}$ so that $E \subset \bigcup_\ell 5B_\ell$. Thus, because $q/p \geq 1$, 
\begin{align*}
  \lambda & \leq \diam(E)/R \lesssim \frac{1}{R}\sum_\ell 5\cdot 2^{-r_\ell+2}
  \lesssim \frac{1}{R}\sum_\ell \left( \int_{B(x_\ell,2^{-r_\ell+2})} g_{r_\ell}(z)^p dz \right)^{q/p}R \\
& = \sum_{k\in\integer} \sum_{\ell : r_\ell = k} \left( \int_{B(x_\ell,2^{-k+2})} g_{k}(z)^p dz \right)^{q/p}
  \leq \sum_{k\in\integer} \left( \sum_{\ell : r_\ell = k} \int_{B(x_\ell,2^{-k+2})} g_{k}(z)^p dz \right)^{q/p} \\
& \leq \sum_{k \in \integer} \left( \int_{\real^n} g_k(z)^p dz \right)^{q/p} \approx \| u \|_{\dot{B}_{n/s,q}^s(\real^n)}^q.\qedhere
\end{align*}
\end{proof}

\begin{remark2}
  Note that the assumptions on the diameters and connectedness of $E$ and $F$ were used only to establish the estimates $\lambda \leq \diam(E)/R \lesssim \frac1R \sum_i r_i$ for any collection of balls $B_i := B(x_i,r_i)$ such that $E \subset \cup_i B_i$ and $x_i \in E$ (and similarly for $F$).
\end{remark2}

Here is an upper estimate for the Besov capacity of an annulus. The result itself is basically well known, but we have included a proof that can for all $q > 1$ be easily adapted to Ahlfors regular metric spaces as well. A similar result for compact Ahlfors $Q$-regular metric spaces and $p = q > Q$ is obtained in \cite[Lemma 1.7]{B}.

\lem{le:capacity-upper}
Let $x_0 \in \real^n$, $0 < r < R < \infty$ and $1 < q \leq \infty$. Then there exists a continuous function $u : \real^n \to \real$ with $u_{|\real^n \backslash B(x_0,R)} = 0$, $u_{|\closure[2]{B}(x_0,r)} = 1$ and
\[
  \|u\|_{\dot{B}_{n/s,q}^s(\real^n)} \leq \Psi(R/r).
\]
Here $\Psi$ is a decreasing homeomorphism from $(1,\infty)$ onto $(0,\infty)$, independent of $x_0$, $r$ and $R$.
\elem

\begin{proof}
We will work with the characterization of the Besov norm employed in lemmas \ref{le:balls-different-sizes} and \ref{le:balls-same-sizes}. For $j \geq 0$, define the balls $B_j := B(x_0,2^{-j}R)$ and the functions $f_j : \real^n \to \real$ with $f_j(y) := (j+1)^{-1}\max(0 , 1 - |x_0-y|/(2^{-j}R))$. We are now in the situation of lemma $\ref{le:balls-different-sizes}$ (i) (with $b_j = (j+1)^{-1}$), so with $F := \sum_j f_j$, we have 
\[
  \|F \|_{\dot{B}_{n/s,q}^s(\real^n)} \leq c\cdot \left( \sum_{j \geq 0} (j+1)^{-q} \right)^{1/q} =: c'
\]
with constants $c$ and $c'$ independent of $x$ and $R$. Note that $F$ is radially decreasing with respect to $x_0$ (because each $f_j$ is). Denoting by $\xi(s)$ the value of $F$ on any sphere of the form $\{ y : |x-y| = sR\}$, $0 < s < 1$, we thus have $\xi(s) > \xi(s')$ for $0 < s < s' < 1$.
In particular, defining $u : \real^n \to \real$ with
\[
  u(x) = \min(1, F(x)/\xi(r/R))
\]
we get a continuous function with $u_{|\real^n \backslash B(x_0,R)} = 0$, $u_{|\closure[2]{B}(x_0,r)} = 1$ and
\[
  \|u\|_{\dot{B}_{n/s,q}^s(\real^n)} \leq \|F\|_{\dot{B}_{n/s,q}^s(\real^n)}/\xi(r/R) \leq c'/\xi(r/R).
\]
It is obvious that the function $s \mapsto \xi(s)$ is independent of $x_0$ and $R$, and that it tends to $0$ as $s \to 1$. To see that it tends to $+\infty$ as $s \to 0$, note that for small $s > 0$ we have
\begin{align*}
  \xi(s) & = \sum_{j \geq 0} (j+1)^{-1} \max\left(0, 1 - 2^j s\right) \geq
  \frac12 \sum_{j=0}^{-\log_2(2s)} (j+1)^{-1} \stackrel{s\to 0^+}{\longrightarrow} \infty.
\end{align*}
The function $\Psi := s \mapsto c'/\xi(1/s)$ therefore satisfies the desired property.
\end{proof}

Combining these two estimates gives the following result, shown in the case of $q = n/s \geq n+1$ in \cite{Vo} (as well as in the settings of $\real$ and compact Ahlfors regular metric spaces in \cite{BS} and \cite{B} respectively).

\prop{pr:qc-necessity}
Suppose that $\varphi : \real^n \to \real^n$ is a homeomorphism such that the mapping induced by $\varphi$ is bounded on $\dot{B}_{n/s,q}^s(\real^n)$ for some $0 < s < 1$ and $ 1 < q < \infty$. Then $\varphi$ is quasiconformal.
\eprop

\begin{proof}
We will show that $\varphi$ actually satisfies the following weak quasisymmetry type condition for some constant $H > 0$:
\[
  \frac{|x-y|}{|x-z|} \leq 1 \qquad {\rm implies} \qquad
  \frac{|\varphi(x) - \varphi(y)|}{|\varphi(x) - \varphi(z)|} \leq H.
\]
This implies $\eta$-quasisymmetry for some homeomorphism $\eta : [0,\infty)\to[0,\infty)$ quantitatively (not only in $\real^n$ but in all pathwise connected doubling metric spaces); see \cite[Chapter 4]{HK} and the references therein. 

To this end, suppose that $0 < |x-y| \leq |x-z|$ and $|\varphi(x) - \varphi(y)| > |\varphi(x) - \varphi(z)|$. With $R := |\varphi(x) - \varphi(y)|$ and $r := |\varphi(x) - \varphi(z)|$, lemma \ref{le:capacity-upper} yields a function $u \in \dot{B}_{n/s,q}^{s}(\real^n)$ with $u_{|\real^n \backslash B(\varphi(x),R)} = 0$, $u_{|\closure[2]{B}(\varphi(x),r)} = 1$ and $\|u\|_{ \dot{B}_{n/s,q}^{s}(\real^n) } \leq \Psi(R/r)$.

On the other hand, let $\gamma_1 := [\varphi(x) , \varphi(z)]$, and let $\gamma_2 \subset \real^n \backslash B(\varphi(x),R)$ be a path joining $\varphi(y)$ to $\varphi(y')$, where $y' \in \real^n \backslash B(x,2|x-z|)$. Furthermore, let $E$ and $F$ be the components of $\varphi^{-1}(\gamma_1) \cap \closure[3]{B}(x,|x-z|)$ and $\varphi^{-1}(\gamma_2) \cap \closure[3]{B} (x,2|x-z|)$ containing $x$ and $y$, respectively. Then $E$ and $F$ are compact and connected subsets of $B(x,3|x-z|)$
with $(u\circ \varphi)_{|E} = 1$, $(u\circ \varphi)_{|F} = 0$ and
\[
  \min(\diam E, \diam F) \gtrsim 3|x-z|,
\]
so Lemma \ref{le:capacity-lower} tells us that $\|u\circ \varphi\|_{ \dot{B}_{n/s,q}^{s}(\real^n) } \gtrsim 1$. Combining these estimates with the boundedness assumption, we get $\Psi(R/r) \geq c$ for some constant $c > 0$ independent of $x$, $y$ and $z$. This in turn yields $R/r \leq \Psi^{-1}(c)$, which is the desired conclusion.
\end{proof}

We are now basically ready for the proof of the main result, but before that we will still record a basic lemma:

\lem{le:jacobian-bounded}
If a homeomorphism $\varphi : \real^n \to \real^n$ is quasisymmetric and $C^{-1} \leq |J_\varphi|  \leq C$ almost everywhere for some constant $C > 1$, then $\varphi$ is bi-Lipschitz.
\elem
\begin{proof}
Take $x$ and $y$ in $\real^n$, $x \neq y$, and write $r := |x-y|$ and $r' := |\varphi(x) - \varphi(y)|$. Since $\varphi$ is quasisymmetric, basic geometric considerations show that $|\varphi(B(x,r))| \approx |B(\varphi(x),r')|$, and since $|J_\varphi| \approx 1$ almost everywhere, we have
\[
  (r')^n \approx \int_{B(\varphi(x),r')} dz \approx \int_{\varphi(B(x,r))} |J_\varphi(z)| dz \approx r^n. \qedhere
\]
\end{proof}

\begin{proof}[{\bf Proof of Theorem \ref{th:bilip-necessity}}]
Write $p = n/s$ as usual. By Proposition \ref{pr:qc-necessity}, $\varphi$ is quasiconformal, so in particular it is quasisymmetric. To show that it is bi-Lipschitz, it therefore suffices to show that
\[
  C^{-1} \leq |J_\varphi|  \leq C \quad \text{ almost everywhere}
\]
for some constant $C > 1$. This will follow once we can show that $|A_k| > 0$ for only finitely many $k$, where
\[
  A_k := \{2^{-n(k+1)}\leq |J_\varphi| < 2^{-nk} \}  \quad \text{ for } k\in\integer.
\]
To this end, suppose that $|A_{k_j}| > 0$ for some sequence $k_1 < k_2 < \cdots < k_N$ of integers. For all $j$, we can find density points $x_j$ of $A_{k_j}$ such that, in addition, $x_j \in A_{k_j}$ and $\lim_{r\to 0^+} |\varphi(B(x_j,r))|/|B(x_j,r)| = |J_\varphi(x_j)|$. We can (by the quasisymmetry of $\varphi$) then take $r > 0$ small enough (depending on $N$) so that for $1 \leq j \leq N$,
\begin{itemize}\setlength{\itemsep}{5pt}
\item  $B\left( \varphi(x_j) , c_\varphi 2^{-k_j} r \right) \subset \varphi\left( B(x_j, r) \right) \subset B\left( \varphi(x_j) , C_\varphi 2^{-k_j} r \right)$,
\item the balls $9 B(x_j , r)$ are pairwise disjoint,
\item the balls $9 B\left(\varphi(x_j), C_\varphi 2^{-k_j} r \right)$ are pairwise disjoint and
\item $\left|B(x_j, r) \cap A_{k_j} \right| \geq \lambda \left|B(x_j, r) \right|$, where $0 < \lambda < 1$ will be fixed later.
\end{itemize}
Let $b_j$, $1 \leq j \leq N$, be positive numbers and define
\[
  g_j(x) := b_j \max\left(0 , 1 - \frac{|x - \varphi(x_j)|}{c_\varphi 2^{-k_j}r} \right)
\]
so that $\supp g_j \subset B\left( \varphi(x_j) , c_\varphi 2^{-k_j} r \right)$ and $\Lip g_j \leq b_j 2^{k_j} (c_\varphi r)^{-1}$. Writing $G := \sum_{1 \leq j \leq N} g_j$, Lemma \ref{le:balls-different-sizes} (i) (with the sequence augmented with zeroes so that each ``$b_j$'' here corresponds to the $2^{k_j}$th number) yields $\|G\|_{\dot{B}_{p,q}^s(\real^n)} \lesssim \|(b_j)_{1\leq j\leq N}\|_{\ell^q}$.

On the other hand, writing $h_j := g_j \circ \varphi$ we have $\supp h_j \subset B(x_j , r)$, so if $H := G\circ \varphi$, lemma \ref{le:balls-same-sizes} (ii) yields $\|H\|_{\dot{B}_{p,q}^s(\real^n)} \gtrsim \| \left( r^{-s}\|h_j\|_{L^p(\real^n)}\right)_{1 \leq j \leq N} \|_{\ell^p} $. In order to estimate $\|h_j\|_{L^p(\real^n)}$, note that
\begin{align*}
  \|h_j\|_{L^p(\real^n)}^{p}
& \geq \int_{B(x_j,r)\cap A_{k_j}} (g_j \circ \varphi)(x)^p dx
  \approx 2^{k_j n} \int_{B(x_j,r)\cap A_{k_j}} (g_j \circ \varphi)(x)^p |J_\varphi(x)| dx \\
& = 2^{k_j n} \int_{\varphi\left(B(x_j,r)\cap A_{k_j}\right)} g_j(x)^p dx.
\end{align*}
We recall that the pullback measure induced by $\varphi$ has the $A_\infty$ property (with respect to the Lebesgue measure on $\real^n$), i.e.~for any $\varepsilon > 0$, there is a $\delta > 0$ so that whenever $B \subset \real^n$ is a ball and $E$ is a measurable subset of $B$ with $|E| < \delta |B|$, one has $|\varphi(E)| < \varepsilon |\varphi(B)|$; see e.g.~\cite[Corollary 4.17]{K}. In particular, since $g_j \geq b_j/2$ on a subset of $\varphi\left(B(x_j,r)\right)$ of measure $\approx (2^{-k_j} r)^n$, $|g_j| \leq b_j$ everywhere and $|B(x_j,r)\backslash A_{k_j}| \leq (1-\lambda) |B(x_j,r)|$, we may take $\lambda$ sufficiently close to $1$ so that
\[
  \int_{\varphi\left(B(x_j,r)\cap A_{k_j}\right)} g_j(x)^p dx
  \gtrsim (2^{-k_j} r)^n b_j^p.
\]
Thus, $\|h_j\|_{L^p(\real^n)} \gtrsim r^{n/p} b_j = r^{s} b_j$, which yields
$
  \|H\|_{\dot{B}_{p,q}^s(\real^n)} \gtrsim \| \left( b_j\right)_{1 \leq j \leq N} \|_{\ell^p}.
$
Recalling that $H = G\circ \varphi$, we therefore have
\beqla{eq:lplq-estimate}
  \| \left( b_j\right)_{1 \leq j \leq N} \|_{\ell^p} \lesssim \|(b_j)_{1\leq j\leq N}\|_{\ell^q}.
\eeq

We shall next establish the converse estimate. Now let the $x_j$ and $b_j$ be as above, but this time choose $r > 0$ small enough (again, depending on $N$) so that for $1 \leq j \leq N$,
\begin{itemize}\setlength{\itemsep}{5pt}
\item $B\left(\varphi(x_j), c_\varphi r\right) \subset \varphi\left( B(x_j, 2^{k_j} r) \right) \subset B\left( \varphi(x_j), C_\varphi r \right)$,
\item the balls $9 B(x_j, 2^{k_j} r)$ are pairwise disjoint,
\item the balls $9 B\left( \varphi(x_j), C_\varphi r \right)$ are pairwise disjoint and
\item $\left|B(x_j, 2^{k_j} r) \cap A_{k_j} \right| \geq \lambda \left|B(x_j, 2^{k_j} r) \right|$, where $0 < \lambda < 1$ will again be fixed later.
\end{itemize}
Define $ g_j(x) := b_j \max\left(0 , 1 - |x - \varphi(x_j)|/(c_\varphi r)\right)$, $h_j := g_j \circ \varphi$ and $G$ and $H$ accordingly as above. Lemma \ref{le:balls-same-sizes} (i) yields  $\|G\|_{\dot{B}_{p,q}^s(\real^n)} \lesssim \|(b_j)_{1\leq j\leq N}\|_{\ell^p}$, and lemma \ref{le:balls-different-sizes} (ii) (with $R = 2^{k_N}r$) yields
\[
  \|H\|_{\dot{B}_{p,q}^s(\real^n)}
  \gtrsim \left\| \left(2^{(k_N - k_j)s} (2^{k_N}r)^{-s} \|h_j\|_{L^{p}(\real^n)}\right)_{1\leq j\leq N} \right\|_{\ell^q}.
\]
Using the $A_\infty$ property and choosing $\lambda$ sufficiently close to $1$ in the same way as above, we get $\|h_j\|_{L^{p}(\real^n)} \gtrsim (2^{k_j} r)^{n/p} b_j = (2^{k_j} r)^s b_j$, so that
$
  \|H\|_{\dot{B}_{p,q}^s(\real^n)} \gtrsim \| \left( b_j\right)_{1 \leq j \leq N} \|_{\ell^q},
$
and by the boundedness assumption, 
$\| \left( b_j\right)_{1 \leq j \leq N} \|_{\ell^q} \lesssim \|(b_j)_{1\leq j\leq N}\|_{\ell^p}$.
Combining this with the converse estimate \refeq{eq:lplq-estimate} above, we get
\[
  \| \left( b_j\right)_{1 \leq j \leq N} \|_{\ell^q} \approx \|(b_j)_{1\leq j\leq N}\|_{\ell^p}
\]
with the implied constants independent of the number $N$ of levels $k_j$ with $|A_{k_j}| > 0$ as well as the numbers $b_j > 0$. Since $p \neq q$, it is thus easily seen that $N$ must be bounded from above by some constant depending only on the dimension $n$, the parameters $s$ and $q$ and the norm of the operator induced by $\varphi$.
\end{proof}

We end this section with a few remarks concerning the case $q \leq 1$. For $q > 1$, Lemma \ref{le:capacity-upper} essentially implies the existence of a function in $\dot{B}^s_{n/s,q}(\real^n)$ which blows up in the vicinity of some point. When $0 < q \leq 1$, such functions do not exist, because every function in $\dot{B}^s_{n/s,q}(\real^n)$ agrees almost everywhere with a bounded and (absolutely) continuous function. This result can be found in \cite[Sections 2.8.3 and 5.2.5]{T}, but it can also be seen fairly easily by an argument similar to the proof of Lemma \ref{le:capacity-lower} (namely, combining the inequalities \refeq{eq:linf-embedding} and \refeq{eq:linf-embedding-2} essentially gives $|u(x)-u(y)| \lesssim \sum_{2^{-k} \leq c|x-y|} \|g_k\|_{L^{n/s}}$ for Lebesgue points $x$ and $y$ of $u$, where $(\|g_k\|_{L^{n/s}})_{k\in\integer} \in \ell^1$).

However, under an a priori regularity assumption on $\varphi$ in the spirit of Reimann \cite{R}, the bi-Lipschitz necessity can also be shown for $q \leq 1$. In particular, it is enough to assume that $\varphi$ is in $W^{1,1}_{\rm{loc}}(\real^n)$. We shall show that, in this case, the boundedness of our composition operator on $\dot{B}^s_{n/s,q}(\real^n)$ for $0 < q \leq 1$ implies the quasiconformality of $\varphi$, and then use interpolation to deduce the bi-Lipschitz regularity of $\varphi$.

To this end, we will make use of the following basic lemma:

\lem{le:small-q}
Let $0 < s < 1$ and $0 < q \leq n/s$. There exists a dimensional constant $c(n) > 0$ with the following property:
For any continuous function $u : \real^n \to \real$ such that $u \geq 1$ on $Q \cap Z$ and $u \leq 0$ on $(3Q \backslash 2Q) \cap Z$,
where $Q = [-A_1,A_1]\times \cdots \times [-A_n,A_n]$ $0 < A_i < \infty$ for $1 \leq i \leq n$ and $Z$ is any set with
\beq\label{eq:small-q-density}
 |3Q \backslash Z| \leq c(n) \min(A_i)^n,
\eeq
we have
\[
  \|u\|_{\dot{B}^s_{n/s,q}(\real^n)} \geq c \left( \frac{A_1 A_2 \cdots A_n}{\min(A_i)^n} \right)^{s/n},
\]
with $c > 0$ independent of $u$, the $A_i$ and $Z$.
\elem

\begin{proof}
For simplicity of notation, we shall consider the case $n = 3$; the proof for general $n$ follows the same idea. By the monotonicity of $\ell^q$ norms (see \refeq{eq:besov-characterizations}), it suffices to consider the case $q = 3/s =: p$. Furthermore, by scaling invariance, we may normalize the situation to $1 = A_1 \leq A_2$, $A_3$. We assume that \refeq{eq:small-q-density} holds for some $c(n)$, which will be chosen more precisely later.

Now, set $Q_{k,\ell} := [1/2,5/2]\times [k-A_2,(k+1)-A_2] \times [\ell-A_3,(\ell+1)-A_3]$ for $0 \leq k < 2\lfloor A_2 \rfloor$, $0 \leq \ell < 2\lfloor A_3 \rfloor$. Putting $Q^{(1)}_{0,0} := ([1/2,1]\times \real^2)\cap Q_{0,0}$ and $Q^{(2)}_{0,0} := ([2,5/2]\times \real^2)\cap Q_{0,0}$, we have $|Q^{(i)}_{0,0}| = 1/2 = \min(A_i)^3/2$, so $c(n)$ can be chosen small enough to guarantee that $|Q^{(i)}_{0,0} \cap Z| \geq 1/4$. For $i = 1$, this can be written out as
\[
  \iint_{[-A_2,1-A_2]\times [-A_3,1-A_3]} \left( \int_{\frac12}^1 \chi_Z (x_1,x_2,x_3) dx_1 \right) dx_2 dx_3 \geq 1/4,
\]
so there exists a point $(y_2,y_3) \in [-A_2,1-A_2]\times [-A_3,1-A_3]$ such that $H^1(E) \gtrsim 1$ for $E := ([1/2,1]\times \{(y_2,y_3)\})\cap Z \subset Q \cap Z$. One may similarly find a point $(z_2,z_3) \in [-A_2,1-A_2]\times [-A_3,1-A_3]$ such that $F := ([2,5/2]\times \{(z_2,z_3)\})\cap Z \subset (3Q \backslash 2Q) \cap Z$ has $H^1$-measure $\gtrsim 1$.

We now note that $Q_{0,0}$ can be viewed as an Ahlfors regular metric measure space (with the Ahlfors regularity constants depending only on the dimension), and we shall make use of a metric version of Lemma \ref{le:capacity-lower} (see Lemma \ref{le:capacity-lower-m} in the following section). Since $E$ and $F$ each lie on a line, it is quite easy to see that we have $H^{1}(E) \lesssim \sum_i r_i$ for any collection of balls $B(x_i,r_i)$ as in the remark after the proof of Lemma \ref{le:capacity-lower} (and similarly for F), so we have
\[
  \| u_{| Q_{0,0}} \|_{\dot{B}^s_{p,p}(Q_{0,0})} \geq c > 0.
\]
The same estimate holds for any $Q_{k,\ell}$ in place of $Q_{0,0}$ as well. Since the interiors of the orthotopes $Q_{k,\ell}$ are pairwise disjoint, we therefore get
\begin{align*}
\| u \|_{\dot{B}^s_{p,p}(\real^3)}^p & \approx \iint_{\real^3 \times \real^3} \frac{|u(x)-u(y)|^p}{|x-y|^{2\cdot 3}}dx dy \\
& \geq \sum_{k=0}^{2\lfloor A_2 \rfloor-1} \sum_{\ell=0}^{2\lfloor A_3 \rfloor-1} \iint_{Q_{k,\ell}\times Q_{k,\ell}} \frac{|u(x)-u(y)|^p}{|x-y|^{2\cdot 3}}dx dy \\
& \geq \sum_{k=0}^{2\lfloor A_2 \rfloor-1} \sum_{\ell=0}^{2\lfloor A_3 \rfloor-1} c^p \approx 2\lfloor A_2 \rfloor \cdot 2\lfloor A_3 \rfloor \geq A_2 A_3. \qedhere
\end{align*}
\end{proof}

\begin{remark2}
Examining the proof above shows that if $A_1 = \min(A_i)$, the assumption $u \leq 0$ can be relaxed to hold on $3Q \cap \left([2A_1,(5/2)A_1] \times \real^{n-1}\right)\cap Z$.
\end{remark2}

The plan is now to use the regularity assumption on $\varphi$ stated above to localize the situation into subsets of $\real^n$ where $\varphi$ behaves roughly like a linear mapping:

\thm{th:small-q2}
Suppose that $\varphi$ in $W^{1,1}_{\rm{loc}}(\real^n)$ and that the mapping induced by $\varphi$ is bounded on $\dot{B}^s_{n/s,q}(\real^n)$ for some $0 < s < 1$ and $0 < q \leq 1$. Then $\varphi$ is bi-Lipschitz.
\ethm
\begin{proof}
Recall that the Sobolev regularity of $\varphi$ implies approximate differentiability almost everywhere, i.e.~that for almost every $x\in \real^n$, $x$ is a density point of
\[
  Z_{x,\varepsilon} := \left\{y : \frac{|\varphi(y) - \varphi(x) - D\varphi(x)(y-x)|}{|y-x|} < \varepsilon\right\}
\]
for all $\varepsilon > 0$; see e.g.~\cite[Section 6.1.3]{EG}.

We shall first establish the quasiconformality of $\varphi$. It is enough to establish the estimate $|D\varphi(x)|^n \leq C|J_\varphi(x)|$ for all points $x$ where $\varphi$ is approximately differentiable (here $|D\varphi|$ stands for the operator norm of the Jacobian matrix of $\varphi$). It obviously suffices to consider the points $x$ where $D\varphi (x) \neq 0$, and by shift invariance we may assume that $x = \varphi(x) = 0$.

Now by the singular value decomposition we may write $D\varphi(0) = U\Sigma V$, where $U$ and $V$ are isometries of $\real^n$ and $\Sigma$ is a diagonal matrix with diagonal entries $\sigma_1 \geq \sigma_2 \geq \cdots \geq \sigma_n \geq 0$. Our estimate shall be
\beqla{qc-estimate}
  \frac{|D\varphi(0)|^n}{|J_\varphi(0)|} \lesssim \|\varphi\|^{n/s},
\eeq
where $\|\varphi\|$ stands for the operator norm of the mapping induced by $\varphi$ and the implied constant is independent of $\varphi$. Note that for $\psi := U^{-1}\varphi(V^{-1}\cdot)$ we have $|D\varphi(0)| = |D\psi(0)|$, $|J_\varphi(0)| = |J_\psi(0)|$ and $\|\varphi\| = \|\psi\|$, so it suffices to establish \refeq{qc-estimate} for $\psi$ instead of $\varphi$. In particular, we have $\psi(x) = \Sigma x + |x|\epsilon(x)$ with $|\epsilon(x)| < \varepsilon$ whenever $x \in V Z_{0,\varepsilon}$.

We will first establish the desired estimate with the additional assumption that $\sigma_1\sigma_2\cdots\sigma_n \neq 0$. We fix a Lipschitz function $f$ with $\chi_{[-5/4,5/4]^n} \leq f \leq \chi_{\real^n \backslash (-7/4,7/4)^n}$ and write $f_h = f(h^{-1}\cdot)$ for $h > 0$. In particular,
\[
  \|f\|_{\dot{B}^s_{n/s,q}(\real^n)} = \|f_h\|_{\dot{B}^s_{n/s,q}(\real^n)}
\]
for all $h$. By approximate differentiability, we may pick $\delta > 0$ so that we have \refeq{eq:small-q-density} for $A_i := \delta \sigma_i^{-1}$ and $Z := V Z_{0,(12\sqrt{n})^{-1}\sigma_n}$. Since $\Sigma$ maps
\[
  Q_{\delta} := \left[-\frac{\delta}{\sigma_1},\frac{\delta}{\sigma_1} \right] \times \cdots \times \left[-\frac{\delta}{\sigma_n},\frac{\delta}{\sigma_n} \right]
\]
into $[-\delta,\delta]^n$ and $3Q_\delta\backslash 2Q_\delta$ into $\real^n \backslash [-2\delta,2\delta]^n$, we have
\[
  \psi\left(Q_\delta \cap Z\right) \subset \left[-\frac{5\delta}{4},\frac{5\delta}{4}\right]^n \quad \rm{and} \quad \psi\left( (3Q_\delta\backslash 2Q_\delta)\cap Z\right) \subset \real^n \backslash \left(-\frac{7\delta}{4},\frac{7\delta}{4}\right)^n.
\]
In particular, $f_\delta \circ \psi$ satisfies the assumptions of Lemma \ref{le:small-q} (with the $A_i$ and $Z$ as above), so
\begin{align*}
\|\psi\| & \approx \|\psi\| \|f_\delta\|_{\dot{B}^s_{n/s,q}(\real^n)} \geq \|f_\delta \circ \psi\|_{\dot{B}^s_{n/s,q}(\real^n)} \gtrsim \left(\frac{\sigma_1^n}{\sigma_1\cdots \sigma_n} \right)^{s/n} = \left(\frac{|D\psi(0)|^n}{|J_\psi(0)|}\right)^{s/n}.
\end{align*}

We shall then show that we must necessarily have $\sigma_1\sigma_2\cdots\sigma_n \neq 0$, so that the additional assumption made above in fact holds in all possible cases. Assuming the contrary, we have $\sigma_i = 0$ when $i > i_0$ and $\sigma_i > 0$ when $i \leq i_0$ for some index $i_0 < n$. For an arbitrary $\varepsilon \in (0,\sigma_{i_0})$, take $\delta > 0$ so that \refeq{eq:small-q-density} holds with $A_i := \delta \sigma_i^{-1}$ for $i \leq i_0$, $A_i := \delta\varepsilon^{-1}$ for $i > i_0$ and $Z := VZ_{0,(12\sqrt{n})^{-1}\varepsilon}$. Now $\Sigma$ maps
\[
  Q_{\varepsilon,\delta} := \left[-\frac{\delta}{\sigma_1},\frac{\delta}{\sigma_1} \right] \times \cdots \times \left[-\frac{\delta}{\sigma_{i_0}},\frac{\delta}{\sigma_{i_0}} \right] \times \left[-\frac{\delta}{\varepsilon},\frac{\delta}{\varepsilon}\right]^{n-i_0}
\]
into $[-\delta,\delta]^n$ and $3Q_{\varepsilon,\delta} \cap \left([2\delta/\sigma_1,(5/2)\delta/\sigma_1] \times \real^{n-1}\right) =: Q'_{\varepsilon,\delta}$ into $[2\delta,(5/2)\delta]\times \real^{n-1}$. Thus,
\[
  \psi\left(Q_{\varepsilon,\delta} \cap Z \right) \subset \left[-\frac{5\delta}{4},\frac{5\delta}{4}\right]^n \quad \rm{and} \quad \psi\left( Q'_{\varepsilon,\delta} \cap Z \right) \subset \left(\frac{7\delta}{4},\infty\right)\times \real^{n-1},
\]
so $f_{\delta} \circ \psi$ satisfies the conditions of the remark after Lemma \ref{le:small-q}'s proof (with the $A_i$ and $Z$ as above), and as before we get
\[
  \|\psi\| \gtrsim \left(\frac{\sigma_1^n}{\sigma_1\cdots\sigma_{i_0} \varepsilon^{n-i_0}}\right)^{s/n}.
\]
Since $n - i_0 > 0$, taking $\varepsilon \to 0^{+}$ yields a contradiction with the boundedness assumption.

All in all, $\varphi$ is quasiconformal, and so the mapping induced by $\varphi$ is bounded on $\dot{B}^{t}_{n/t,n/t}(\real^n)$ for all $t \in (0,1)$. In particular, interpolating between such a space and $\dot{B}^s_{n/s,q}(\real^n)$ (see ~\cite[Theorems 9.1 and 8.1]{KMM}), we obtain boundedness on $\dot{B}^{s'}_{n/s',q'}(\real^n)$ for some $s' \in (0,1)$ and $q' \in (1,n/s')$. Theorem \ref{th:bilip-necessity} thus implies the bi-Lipschitz regularity of $\varphi$.
\end{proof}

\begin{remark2}
The proof above makes use of complex interpolation of function spaces which are not Banach spaces. We refer to \cite{KMM} for a treatment of the interpolation of the spaces in question.
\end{remark2}

Another approach to the case $q \leq 1$ is in the spirit of Astala \cite{A}: namely, instead of Sobolev regularity, we assume that the mapping induced by $\varphi$ is uniformly bounded from $\dot{B}^s_{n/s,q}(\Omega)$ into $\dot{B}^s_{n/s,q}(\varphi^{-1}(\Omega))$ for all bounded domains $\Omega \subset \real^n$. By a Besov space on a domain $\Omega$ we mean the space (quasi-)normed by Haj\l asz $s$-gradients (conditions \refeq{eq:besov-characterizations} and \refeq{eq:besov-characterizations-2} with $\Omega$ in place of $\real^n$).

\thm{th:small-q-3}
Suppose that $0 < q \leq 1$ and $\varphi$ is as above. Then $\varphi$ is bi-Lipschitz.
\ethm

\begin{proof}
As usual, let $p = n/s$. To show the quasiconformality of $\varphi$, it suffices to show that
\beqla{eq:qs-condition}
  \frac{|\varphi (B)|}{\left(\diam \varphi (B)\right)^n} \geq c > 0
\eeq
for all open balls $B \subset \real^n$; see \cite[p.~64]{HK2}. To this end, let $\omega_1$ and $\omega_2$ be boundary points of $\Omega := \varphi(B)$ with distance $\diam \Omega$.
Letting
$
  u(z) := |z - \omega_1|/\diam \Omega
$
for $z \in \Omega$, an easy computation shows that $u \in \dot{B}^s_{p,q}(\Omega)$ with
\[
  \|u\|_{\dot{B}^s_{p,q}(\Omega)} \leq C \left(\frac{|\Omega|}{\left(\diam \Omega\right)^n} \right)^{s/n}
\]
(take $g_k = 2^{-k(1-s)}\chi_{\Omega} / \diam(\Omega)$ for $k \geq k_\Omega-1$ and $g_k = 0$ otherwise, where $2^{-k_\Omega} \leq \diam (\Omega) < 2^{-k_\Omega + 1}$).

Now $x_i := \varphi^{-1}(\omega_i)$, $i = 1$, $2$, are boundary points of $B$ to which $u \circ \varphi$ extends continuously as $0$ and $1$ respectively. Let $(g_k)_{k\in\integer}$ be an $s$-gradient of $u\circ \varphi$ with $\sum_k \|g_k\|_{L^{p}(B)}^q \approx \|u\circ \varphi\|_{\dot{B}^s_{p,q}(B)}^q$ and take $\ell \in \integer$ so that $2^{-2\ell+2} \leq |x_1 - x_2| < 2^{-2\ell+4}$. Writing
$
  A_j^{x_i} := A(x_i,2^{-2j-1},2^{-2j}) \cap B
$
for $j \geq \ell$, we have $|A^{x_i}_j| \approx 2^{-2jn}$, so the estimates in the proof of Lemma \ref{le:capacity-lower} yield
\[
  |(u\circ \varphi)_{A^{x_i}_{j}} - (u\circ \varphi)_{A^{x_i}_{j+1}}| \lesssim \sum_{2^{-2j-2} \leq 2^{-k} \leq 2^{-2j+1}} \|g_k\|_{L^{p}(B)},
\]
and summing over $j \geq \ell$ yields
\[
  |(u\circ \varphi)(x_i) - (u\circ \varphi)_{A^{x_i}_{\ell}}| \lesssim \sum_{2^{-k} \leq 2^{-2\ell +1}} \|g_k\|_{L^p(B)} \leq \sum_{2^{-k} \leq |x_1 - x_2|} \|g_k\|_{L^p(B)}.
\]
Also,
\begin{align*}
  |(u\circ \varphi)_{A^{x_1}_\ell} - (u\circ \varphi)_{A^{x_2}_\ell}|
& \lesssim \sum_{2^{-2\ell +1} \leq 2^{-k} \leq 2^{-2\ell +3}} \|g_k\|_{L^p(B)}
& \leq \sum_{2^{-k} \leq 2|x_1 - x_2|} \|g_k\|_{L^p(B)}.
\end{align*}
Altogether we thus have
\[
  |(u\circ \varphi)(x_1) - (u\circ \varphi)(x_2)| \lesssim \sum_{2^{-k} \leq 2|x_1 - x_2|} \|g_k\|_{L^p(B)} \leq \left( \sum_{2^{-k} \leq 2|x_1 - x_2|} \|g_k\|_{L^p(B)}^q \right)^{1/q}.
\]
Noting that the left hand side above is $1$, we thus have \[1 \lesssim \|u\circ \varphi\|_{\dot{B}^s_{p,q}(B)} \lesssim \|u\|_{\dot{B}^s_{p,q}(\Omega)} \lesssim \left(\frac{|\Omega|}{(\diam \Omega)^n} \right)^{s/n},\] showing the quasiconformality of $\varphi$. One can conclude by interpolating as in the proof of Theorem \ref{th:small-q2}.
\end{proof}

\section{Besov spaces on metric spaces}\label{section:metric}

The purpose of this section is to explain how the main results of the previous section extend to the setting of metric spaces with sufficiently reasonable geometry. The main reference for the required facts about such spaces and quasiconformal mappings on them is \cite{HK}, particularly Chapters 4 and 7. For the equivalence of the different Besov type norms, we again refer to \cite{GKZ} and \cite{KYZ2}.

The main results in the previous chapter have natural counterparts in the context of metric spaces, and thanks to our methods, the proofs are the same. Note that in a connected $Q$-Ahlfors regular\footnote{Recall that $Q$-Ahlfors regularity means that $\mu(B_d(x,r)) \approx r^Q$ for $x \in X$ and $0 < r < \diam(X)$.} metric measure space $(X,d,\mu)$, we have $\mu(A_d(x,\lambda/2,\lambda)) \approx \lambda^Q$ for $x \in X$ and $0 < \lambda < \diam X$, since $B_d(y,\lambda/4) \subset A_d(x,\lambda/2,\lambda)$ for any $y \in \partial B_d(x,3\lambda/4)$. In the Lemmas \ref{le:balls-different-sizes-m}--\ref{le:capacity-upper-m} below, $X := (X,d,\mu)$ will be a pathwise connected $Q$-Ahlfors regular metric measure space, $Q > 1$. All measures appearing in this chapter are assumed to be Borel regular.

\lem{le:balls-different-sizes-m}

Let $0 < s < 1$, $1 \leq q \leq \infty$ and $0 < R < \infty$. Suppose that $(B_j)_{j = 0}^\infty$ is a sequence of balls in $X$ such that $B_j$ has radius $2^{-j } R$ for all $j \geq 0$, and that $(f_j)_{j=0}^\infty$ is a sequence of measurable functions on $X$ with $\supp f_j \subset \closure[2]{B_j}$ for all $j \geq 0$. Write $F := \sum_{j\geq 0} f_j$.

(i) If the functions $f_j$ are Lipschitz with
\[
  \Lip f_j \leq 2^{j} R^{-1} b_j, \quad b_j > 0,
\]
for all $j \geq 0$, we have
\[
  \|F\|_{\dot{B}_{Q/s,q}^{s}(X)} \lesssim \| (b_j)_{j \geq 0} \|_{\ell^q}.
\]

(ii) If the balls $9 B_j$ are pairwise disjoint, we have
\[
  \|F\|_{\dot{B}_{Q/s,q}^{s}(X)} \gtrsim \left\| (2^{j s} R^{-s} \|f_j\|_{L^{Q/s}(X)})_{j\geq 0} \right\|_{\ell^q}.
\]
The implied constants in both conclusions are independent of $R$.
\elem

\lem{le:balls-same-sizes-m}

Let $0 < s < 1$, $1 \leq q \leq \infty$ and $0 < R < \infty$. Suppose that $(B_j)_{j=0}^\infty$ is a collection of balls in $X$ such that each ball has radius $R$ and the balls $9 B_j$ are pairwise disjoint. Suppose also that $(f_j)_{j=0}^\infty$ is a sequence of functions on $X$ with $\supp f_j \subset \closure[2]{B_j}$ for all $j \geq 0$. Write $F := \sum_{j \geq 0} f_j$.

(i) If the functions $f_j$ are Lipschitz with
\[
  \Lip f_j \leq R^{-1} b_j, \quad b_j > 0,
\]
for all $j \geq 0$, we have
\[
  \|F\|_{\dot{B}_{Q/s,q}^s(X)} \lesssim \| (b_j)_{j \geq 0} \|_{\ell^{Q/s}}.
\]

(ii) We have (with the functions $f_j$ not necessarily Lipschitz)
\[
  \|F\|_{\dot{B}_{Q/s,q}^s(X)} \gtrsim \left\| (R^{-s}\|f_j\|_{L^{Q/s}(X)})_{j\geq 0}  \right\|_{\ell^{Q/s}}.
\]
The implied constants in both conclusions are independent of $R$.
\elem

\lem{le:capacity-lower-m}
Suppose that $0 < s < 1$, $1 \leq q < \infty$ and that $B_R$ is a ball of radius $R < \diam X$ in $X$. If $0 < \lambda \leq 1$, $E$ and $F$ are two such compact and connected subsets of $B_R$ that
\[
  \min(\diam E, \diam F) \geq \lambda R
\]
and $u$ is a continuous function on $X$ with $u_{|E} \leq 0$ and $u_{|F} \geq 1$, we have
\[
  \| u \|_{\dot{B}_{Q/s,q}^s(X)} \geq c \lambda^{1/q},
\]
where $c > 0$ depends only on $q$, $s$ and the data on $X$.
\elem

\lem{le:capacity-upper-m}
Let $x_0 \in X$, $0 < r < R < \diam X$ and $1 < q \leq \infty$. Then there exists a continuous function $u : X \to \real$ with $u_{|X \backslash B(x_0,R)} = 0$, $u_{|\closure[2]{B}(x_0,r)} = 1$ and
\[
  \|u\|_{\dot{B}_{Q/s,q}^s(X)} \leq \Psi(R/r).
\]
Here $\Psi$ is a decreasing homeomorphism from $(1,\infty)$ onto $(0,\infty)$, independent of $x$, $r$ and $R$.
\elem

Before coming to the main result of this section, we will formulate more precisely the assumptions on the metric spaces under consideration, and note some of their consequences. We refer to \cite{HK} for the necessary definitions (such as those of linear local connectivity and quasiconvexity).

In the three results below, $X := (X,d,\mu)$ is assumed to be a proper, quasiconvex and $Q$-Ahlfors regular ($Q > 1$) metric measure space. $X$ is also assumed to support a weak $(1,Q)$-Poincar\'e inequality. $Y := (Y,d',\nu)$ is assumed to be a pathwise connected, $Q$-Ahlfors regular, locally compact and linearly locally connected (LLC) metric measure space. As $X$ and $Y$ are both locally compact, $\mu$ and $\nu$ are comparable to the $Q$-Hausdorff measures on $X$ and $Y$ respectively (see \cite[Remark 3.4]{HK} and the references therein), so we may in fact assume that $\mu$ and $\nu$ \emph{are} said Hausdorff measures. We assume that either the spaces $X$ and $Y$ are both bounded or they are both unbounded.

The fact that $X$ is proper means that it is also complete, so by the main result of \cite{KZ} and the Poincar\'e inequality assumption, $X$ supports a weak $(1,p)$-Poincar\'e inequality for some $p < Q$.

The following proposition generalizes the case $q = Q/s$ obtained for compact spaces in \cite{B}.

\prop{pr:qc-necessity-m}
Suppose that $\varphi : X \to Y$ is a homeomorphism such that the mapping induced by $\varphi$ is bounded from $\dot{B}_{Q/s,q}^s(Y)$ to $\dot{B}_{Q/s,q}^s(X)$ for some $0 < s < 1$ and $ 1 < q < \infty$. Then $\varphi$ is quasiconformal.
\eprop

\begin{proof}
What we want to show is that
\[
  \frac{ d'(\varphi(x),\varphi(y)) }{ d'(\varphi(x),\varphi(z)) } \leq H < \infty
\]
whenever $0 < d(x,y) \leq d(x,z)$ and $d(x,z)$ is sufficiently small (depending on $x$). Denoting by $C_L$ the LLC constant of $Y$, it will be enough to consider the case
\[
  \frac{ d'(\varphi(x),\varphi(y)) }{ d'(\varphi(x),\varphi(z)) } > C_L^2.
\]
Writing $R := d'(\varphi(x),\varphi(y))$ and $r := d'(\varphi(x),\varphi(z))$, Lemma \ref{le:capacity-upper-m} implies the existence of a continuous function $u \in \dot{B}_{Q/s,q}^s(Y)$ with $u_{|\closure[2]{B}_{d'}(\varphi(x),C_L r)} = 1$, $u_{|Y\backslash B_{d'}(\varphi(x),R/C_L)} = 0$ and $\|u\|_{\dot{B}_{Q/s,q}^s(Y)} \leq \Psi_{Y}( (R/C_L)/(C_L r) )$. Now, choosing $y' \in X$ so that $d(x,y') \geq 2d(x,z)$ and $d'(\varphi(x),\varphi(y')) \geq R$, the local linear connectivity of $Y$ yields compact and connected sets $\gamma_1 \subset \closure[2]{B}_{d'}(\varphi(x),C_L r)$ and $\gamma_2 \subset Y \backslash B_{d'}(\varphi(x), R/C_L)$ joining $\varphi(x)$ with $\varphi(z)$ and $\varphi(y)$ with $\varphi(y')$ respectively. Letting $E$ and $F$ be the components containing $x$ and $y$ of $\gamma_1^{-1}(E)\cap \closure[2]{B}_d(x, d(x,z))$ and $\gamma_2^{-1}(f)\cap \closure[2]{B}_d(x, 2d(x,z))$ respectively, we get as in the proof of Proposition \ref{pr:qc-necessity} that $\Psi_Y \left( (R/C_L) / (C_L r) \right) \geq c$, i.e.~$R/r \leq C_L^2 \Psi_Y^{-1}(c)$ for some $c > 0$ depending only on $\varphi$ and the data on the metric spaces.
\end{proof}

\begin{remark2}
We point out that the only assumptions on the metric spaces in question that we used in the proof above were the Ahlfors regularity and connectedness of both spaces, as well as the linear local connectivity of $Y$.
\end{remark2}

For the case where $q \neq Q/s$ we will need the absolute continuity properties of quasisymmetric mappings, see \cite[Chapter 7]{HK}. It is here that we use the assumption that $\mu$ and $\nu$ are the Hausdorff measures of $X$ and $Y$ respectively. The fact that quasiconformality implies quasisymmetry follows from our assumptions on the metric spaces as well as the assumption that $\varphi$ maps bounded sets onto bounded sets, see \cite[Theorem 5.7 and Corollary 4.10]{HK}.

\thm{th:bilip-necessity-m}
Suppose that $\varphi : X \to Y$ is a homeomorphism such that the mapping induced by $\varphi$ is bounded from $\dot{B}_{Q/s,q}^s(Y)$ to $\dot{B}_{Q/s,q}^s(X)$ for some $0 < s < 1$ and $ 1 < q < \infty$, $q \neq Q/s$. Suppose also that $\varphi$ maps bounded sets onto bounded sets. Then $\varphi$ is bi-Lipschitz.
\ethm

\begin{proof}
As pointed out above, $\varphi$ is now quasisymmetric. The volume derivative $\mu_\varphi$ of $\varphi$ is defined by
\[
  \mu_\varphi(x) = \lim_{r\to \infty} \frac{\nu\left(\varphi\left(B(x,r)\right)\right)}{\mu\left(B(x,r)\right)}
\]
for almost every $x$. Then we have the usual change of variables formula with $\mu_\varphi$ in the role of ``$|J_\varphi|$'' and the pullback measure induced by $\varphi$ on $X$ is $A_\infty$-related to $\mu$ \cite[Theorem 7.11]{HK}. The proof of Theorem \ref{th:bilip-necessity} (and Lemma \ref{le:jacobian-bounded}) can thus be carried out by choosing
\[
  A_k := \left\{2^{-Q(k+1)} \leq \mu_\varphi < 2^{-Qk}\right\}.\qedhere
\]
\end{proof}

\begin{corollary}\label{co:boundedness-conditions-m}\sl
Suppose that $\varphi : X \to Y$, is a homeomorphism that maps bounded sets onto bounded sets. Then

(i) if $0 < s < 1$, the mapping induced by $\varphi$ is bounded from $\dot{B}_{Q/s,Q/s}^s (Y)$ into $\dot{B}_{Q/s,Q/s}^s (X)$ if and only if $\varphi$ is quasiconformal;

(ii) if $0 < s < 1 < q < \infty$ and $q \neq Q/s$, the mapping induced by $\varphi$ is bounded from $\dot{B}_{Q/s,q}^s (Y)$ into $\dot{B}_{Q/s,q}^s (X)$ if and only if $\varphi$ is bi-Lipschitz.
\end{corollary}
\begin{proof}
The sufficiency for part (i) follows from \cite[Theorem 1.4]{KYZ2}. The rest is contained in the two previous results.
\end{proof}

\section{Triebel-Lizorkin spaces}\label{section:tl}

In this section, we briefly discuss the analogous mapping question for the scaling invariant Triebel-Lizorkin spaces\footnote{We refer to e.g.~\cite[Section 5.1.3]{T} or \cite[Definition 3.1]{KYZ2} for the (Fourier-analytic) definition of $\dot{F}^s_{p,q}(\real^n)$, $0 < p < \infty$, $0 < q \leq \infty$, $s \in \real$, and \cite[Theorem 1.2]{KYZ2} for a pointwise characterization of the spaces with $0 < s < 1$, $n/(n+s) < p < \infty$ and $n/(n+s) < q \leq \infty$.} $\dot{F}^s_{n/s,q}(\real^n)$, $0 < s < 1$, $0 < q \leq \infty$. The dimension $n$ will always be at least $2$. We recall that it is shown in \cite{KYZ2} that quasiconformality is a sufficient condition for the mapping induced by a homeomorphism $\varphi$ of $\real^n$ to be bounded on $\dot{F}^s_{n/s,q}(\real^n)$ when $0 < s < 1$ and $n/(n+s) < q \leq \infty$. Here we provide the converse to this using our results from section \ref{section:real}:

\prop{pr:tl-qc-necessity}
Suppose that the mapping induced by $\varphi$ is bounded on $\dot{F}^s_{n/s,q}(\real^n)$ for some $0 < s < 1$ and $0 < q \leq \infty$. Then $\varphi$ is quasiconformal. 
\eprop
\begin{proof}
We may argue as in the proof of Proposition \ref{pr:qc-necessity} as long as we can verify capacity estimates analogous to Lemmas \ref{le:capacity-lower} (with, say, some fixed $\lambda$) and \ref{le:capacity-upper}. By \cite[Theorem 2.1 (ii)]{J}, we have
\[
  \dot{B}^{s_0}_{n/s_0,n/s_0}(\real^n) = \dot{F}^{s_0}_{n/s_0,n/s_0}(\real^n) \subset \dot{F}^{s}_{n/s,q}(\real^n) \subset \dot{F}^{s_1}_{n/s_1,n/s_1}(\real^n)  = \dot{B}^{s_1}_{n/s_1,n/s_1}(\real^n),
\]
where $0 < s_1 < s < s_0 < 1$ and the embeddings are continuous, so we get the desired capacity estimates by combining these embeddings with Lemmas \ref{le:capacity-lower} and \ref{le:capacity-upper}.
\end{proof}

Combining this with \cite[Theorem 1.1]{KYZ2} we thus have the following result.

\begin{corollary}\label{tl-boundedness-condition}\sl
Suppose that $0 < s < 1$ and $n/(n+s) < q \leq \infty$. Then the mapping induced by a homeomorphism $\varphi$ of $\real^n$ is bounded on $\dot{F}^s_{n/s,q}(\real^n)$ if and only if $\varphi$ is quasiconformal.
\end{corollary}

Next, we consider Triebel-Lizorkin-type function spaces on metric measure spaces $X$ and $Y$ as in the discussion preceding Proposition \ref{pr:qc-necessity-m}. For $s \in (0,1)$ and $1 \leq q \leq \infty$, we denote by $\dot{F}^s_{Q/s,q}(X)$ the space $\dot{M}^s_{Q/s,q}(X)$ defined in \cite[Definition 1.2]{KYZ2}, and similarly for $Y$ instead of $X$. Then it follows from \cite[Theorem 1.4]{KYZ2} that the quasisymmetry of $\varphi$ is a sufficient condition for the mapping induced by a homeomorphism $\varphi : X \to Y$ to be bounded from $\dot{F}^s_{Q/s,q}(Y)$ to $\dot{F}^s_{Q/s,q}(X)$ (\cite[Theorem 1.4]{KYZ2} actually holds for a larger parameter range). We have the following partial converse.

\prop{pr:tl-qc-necessity-m}
Suppose that the mapping induced by a homeomorphism $\varphi:X\to Y$ is bounded from $\dot{F}^s_{Q/s,q}(Y)$ to $\dot{F}^s_{Q/s,q}(X)$ for some $0 < s < 1$ and $1 < q < \infty$. Then $\varphi$ is quasiconformal. 
\eprop

\begin{proof}
We may argue as in the proof of Proposition \ref{pr:qc-necessity-m}; the necessary capacity estimates follow by combining Lemmas \ref{le:capacity-lower-m} and \ref{le:capacity-upper-m} with the continuous embeddings
\[
  \dot{B}^s_{Q/s,\min(Q/s,q)}(X) \subset \dot{F}^s_{Q/s,q}(X) \subset \dot{B}^s_{Q/s,\max(Q/s,q)}(X),
\]
which one may obtain using the Minkowski integral inequality (see also e.g.~\cite[Proposition 2.3.2/2]{T}), and similar embeddings with $Y$ in place of $X$.
\end{proof}

We thus have the following result.

\begin{corollary}\label{tl-boundedness-condition-m}\sl
Let $\varphi:X\to Y$ be a homeomorphism that maps bounded sets onto bounded sets. If $0 < s < 1$ and $1 < q < \infty$, the mapping induced by $\varphi$ is bounded from $\dot{F}^s_{Q/s,q}(Y)$ to $\dot{F}^s_{Q/s,q}(X)$ if and only if $\varphi$ is quasiconformal.
\end{corollary}

\begin{proof}
By the assumption that $\varphi$ maps bounded sets onto bounded sets, the quasiconformality of $\varphi$ is equivalent with the quasisymmetry of $\varphi$ (see the discussion preceding Theorem \ref{th:bilip-necessity-m}). Thus the statement follows by combining \cite[Theorem 1.4]{KYZ2} and Proposition \ref{pr:tl-qc-necessity-m}.
\end{proof}

\end{document}